\documentclass[10pt]{amsart}

\usepackage[nice]{nicefrac}
\usepackage{amsmath}
\usepackage{amsthm}
\usepackage{amssymb}
\usepackage{delarray}
\usepackage[all,cmtip]{xy}
\usepackage[dvips]{graphics}
\usepackage{epsfig}
\usepackage{mathrsfs}
\usepackage{color}
\usepackage{setspace}  
\usepackage{mathptmx}
\usepackage{stmaryrd}
\usepackage{graphicx}
\usepackage{color}
\usepackage{fancybox}

\numberwithin{equation}{section}

\renewcommand\bar\overline

\newtheorem{thm}{Theorem}[section]
\newtheorem{lemma}[thm]{Lemma}
\newtheorem{prop}[thm]{Proposition}

\theoremstyle{definition}

\theoremstyle{remark}

\begin{document}

\title{A proof of the positive density conjecture for integer Apollonian circle packings}
\author{Jean Bourgain}
\address{Institute for Advanced Study, School of Mathematics, Einstein Drive, Princeton, NJ 08540 USA}
\thanks{The first author is supported in part by NSF}
\email{bourgain@math.ias.edu}
\author{Elena Fuchs}
\address{Princeton University, Department of Mathematics, Fine Hall, Washington Rd, Princeton, NJ 08544-100 USA}
\email{efuchs@math.princeton.edu}

\begin{abstract}
A bounded Apollonian circle packing (ACP) is an ancient Greek construction which is made by repeatedly inscribing circles into the triangular interstices in a Descartes configuration of four mutually tangent circles.  Remarkably, if the original four circles have integer curvature, all of the circles in the packing will have integer curvature as well.  In this paper, we compute a lower bound for the number $\kappa(P,X)$ of integers less than $X$ occurring as curvatures in a bounded integer ACP $P$, and prove a conjecture of Graham, Lagarias, Mallows, Wilkes, and Yan that the ratio $\kappa(P,X)/X$ is greater than $0$ for $X$ tending to infinity.

\end{abstract}

\maketitle
\section{Introduction}\label{Intro}

In the first picture in Figure~\ref{empty} there are three mutually tangent circles packed in a large circle on the outside, with four curvilinear triangles inbetween.  By an old theorem (circa 200 BC) of Apollonius of Perga, there are precisely two circles tangent to all of the circles in a triple of mutually tangent circles.  One can therefore inscribe a unique circle into each curvilinear triangle as in the second picture in Figure~\ref{empty}. Since this new picture has many new curvilinear triangles, we can continue packing circles in this way -- this process continues indefinitely, and we thus get an infinite packing of circles known as the Apollonian circle packing (ACP).  Since the radii of most of the circles in an ACP are extremely small, it is convenient to measure the circles in the packing via their {\it curvatures}, or the reciprocals of the radii, instead.
\begin{figure}[htp]
\centering
\includegraphics[height = 30 mm]{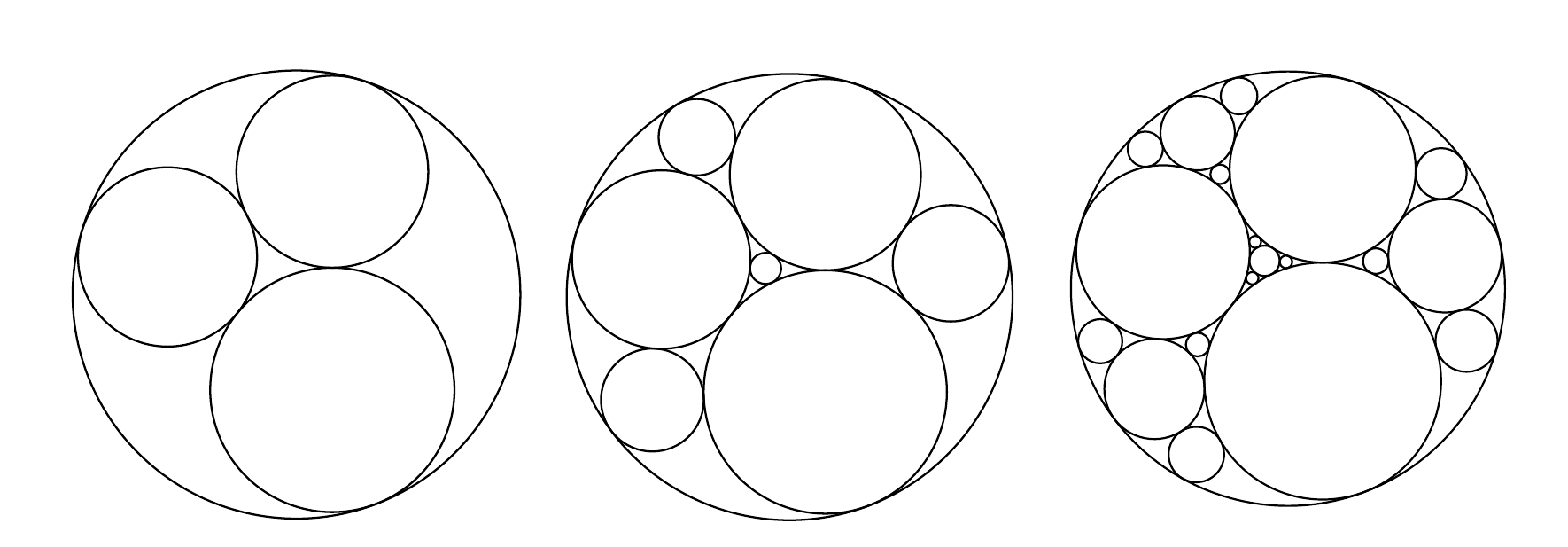}
\caption{Apollonian Circle Packings}\label{empty}
\end{figure}

A remarkable feature of these packings is that, given a packing in which the initial four mutually tangent circles have integer curvature, all of the circles in the packing will have integer curvature as well.  Such a packing is called an integer ACP and is illustrated in Figure~\ref{Drawing} with the packing generated by starting with circles of curvatures $1, 2, 2,$ and $3$.

There are various problems associated to the diophantine properties of integer ACP's which are addressed in \cite{Apollo} by the five authors Graham, Lagarias, Mallows, Wilks, and Yan.  They make considerable progress in treating the problem, and ask several fundamental questions many of which are now solved and discussed further in \cite{sl}, \cite{Fuchs}, \cite{FS}, \cite{KO}, and \cite{maa}.

In most of these papers, ACP's are studied using a convenient representation of the curvatures appearing in an ACP as maximum-norms of vectors in an orbit of a specific subgroup $A$ of the orthogonal group $\textrm{O}(3,1)$.  We introduce this group in Section~\ref{agroup} and will use it throughout.  In regards to counting the number of integers represented in a given ACP, Graham et. al. exploit the existence of unipotent elements of $A$ in \cite{Apollo} to establish the lower bound below for the number $\kappa(P,X)$ of distinct curvatures less than $X$ of circles in an integer packing $P$:
\begin{equation}\label{apbound}
\kappa(P,X)\gg \sqrt{X}
\end{equation}
where the notation 
$$y \gg_{\beta} z\; \mbox{ or }\; y\ll_{\beta} z$$
in this paper is taken to mean that there exists a constant $c>0$ depending only on $\beta$ such that 
$$y\geq cz\; {\mbox{ or, respectively }}\; y\leq cz.$$
Graham et. al. suggest in \cite{Apollo} that the lower bound in (\ref{apbound}) can be improved.  In fact, they conjecture that the integers represented as curvatures in a given ACP actually make up a positive fraction of the positive integers $\mathbb N$.

It is important to note that this question is different from one recently addressed in \cite{KO} by Kontorovich and Oh about the number $N_P(X)$ of circles in a given packing $P$ of curvature less than $X$.  This involves counting curvatures appearing in a packing with multiplicity, rather than counting every integer which comes up exactly once as we do in this paper.  In fact, the results in \cite{KO} suggest that the integers occurring as curvatures in a given ACP arise with significant multiplicity.  Specifically, Kontorovich and Oh find that $N_P(X)$ is asymptotic to $c\cdot x^{\delta}$, where $\delta= 1.3056\dots$ is the Hausdorff dimension of the limit set of the packing.  
Kontorovich and Oh's techniques, however, do not extend in any obvious way to proving that the integers represented by curvatures in an ACP make up a positive fraction in $\mathbb N$.

In \cite{sl} Sarnak uses the existence of arithmetic Fuchsian subgroups of $A$ to get a bound of
\begin{equation}\label{sarnakbound}
\kappa(P,X)\gg \frac{X}{\sqrt{\log{X}}}
\end{equation}
towards Graham.et.al.'s positive density conjecture.  This method, which we summarize in Section~\ref{prlow}, was further improved to yield a bound of
$$\kappa(P,X)\gg \frac{X}{({\log{X}})^{\epsilon}}$$
where $\epsilon= 0.150\dots$ by the second author in a preprint \cite{Fuchspreprint}.

\begin{figure}[htp]
\centering
\includegraphics[height = 70 mm]{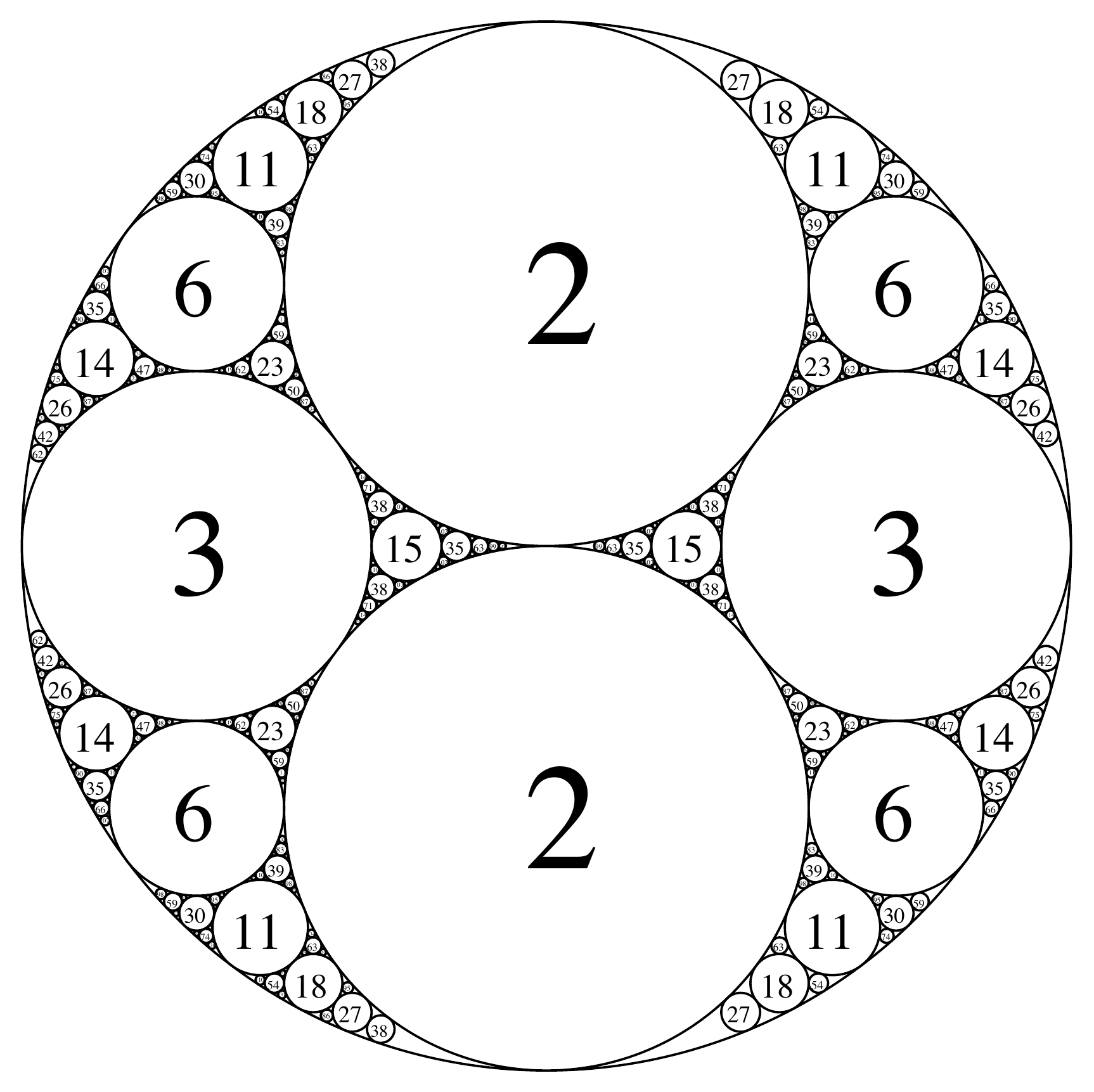}
\caption{Apollonian Circle Packing $(-1,2,2,3)$}\label{Drawing}
\end{figure}

In this paper, we refine this Fuchsian subgroup method in a number of ways and settle the positive density question of Graham et.al. in the following theorem:
\begin{thm}\label{iteratedens}
For an integer Apollonian circle packing $P$, let $\kappa(P, X)$ denote the number of distinct integers up to $X$ occurring as curvatures in the packing.  Then for $X$ large we have
$$\kappa(P, X) \gg X $$
where the implied constant depends on the packing $P$.
\end{thm}
We treat this question by counting curvatures in different ``subpackings" of an ACP.  Namely, we fix a circle $C_0$ of curvature $a_0$ and investigate which integers occur as curvatures of circles tangent to $C$.  This gives the preliminary lower bound in (\ref{sarnakbound}) which was first proven by Sarnak in \cite{sl}.   The essential observation which leads to this lower bound is that the set of integers appearing as curvatures of circles tangent to $C_0$ contain the integers represented by an inhomogeneous binary quadratic form
$$f_{a_0}(x,y)-a_0$$
of discriminant $-4a_0^2$.  Our approach in Section~\ref{iterate} is to repeat this method for a subset of the circles which we find are tangent to $C_0$ in this way.  For every circle $C$ of curvature $a$ tangent to $C_0$, we can produce a shifted binary quadratic form
$$f_a(x,y)-a$$
where $f_a$ has discriminant $-4a^2$ and consider the integers represented by $f_a$.  We consider $a$ in a suitably reduced subset of $[(\log X)^2,(\log X)^3]$ and count the integers represented by $f_a-a$ for $a$ in this subset.  It is important to note that the integers represented by $f_a$ and $f_{a'}$ for $a\not= a'$ are a subset of integers which can be written as a sum of two squares since both forms have discriminant of the form $-\delta^2$.  In fact, $f_a$ and $f_{a'}$ represent practically the same integers (see \cite{Fuchspreprint} for a more detailed discussion).  It is rather the {\it shift} of each form $f_a$ by $a$ that makes the integers found in this way vary significantly.  Our final step is to give an upper bound on the number of integers in the intersection
$$\{m\mbox{ represented by }f_a-a\}\cap\{m'\mbox{ represented by }f_{a'}-a'\}$$
In obtaining this upper bound, we count integers with multiplicity, which is a sacrifice we can afford to make for our purposes.  This method leads to a proof of the conjecture of Graham et.al. that the integers appearing as curvatures in a given integer ACP make up a positive fraction of all integers.\smallskip

{\bf Acknowledgements:}  We thank Peter Sarnak and Jeffrey Lagarias for introducing this problem to us, and for many insightful comments and conversations.  We thank Alex Kontorovich for sharing his program for drawing Apollonian circle packings.

\subsection{The Apollonian group}\label{agroup}
In 1643, Descartes discovered that the curvatures of any four externally cotangent circles of curvatures $x_1$, $x_2$, $x_3$, and $x_4$ satisfy the equation
\begin{equation}\label{descartes}
Q(x_1, x_2, x_3, x_4) = 2(x_1^2 + x_2^2 +x_3^2+x_4^2) - (x_1+x_2+x_3+x_4)^2.
\end{equation}
A proof of this can be found in \cite{Cx}.  This equation also holds for four mutually tangent circles when three of the circles are internally tangent (or inscribed) to the fourth circle as in the case of ACP's.  However, because the outside circle is internally tangent to the other three, we assign to it a negative curvature in order for the equation to hold.  We note that Descartes' theorem allows us to solve a quadratic equation for the fourth curvature once three are known (this is one way to verify Apollonius's theorem about circles tangent to a triple of mutually tangent circles).  

We thus assign to every set of $4$ mutually tangent circles in a packing $P$ a vector $\mathbf v\in \mathbb Z^4$ of the circles' curvatures, and use Descartes' equation to express any ACP as the orbit of a subgroup $A$ of the orthogonal group $\textrm{O}_Q(\mathbb Z)$ acting on $\mathbf v$.  In \cite{Apollo}, Graham et. al. describe an algorithm for generating a {\it root quadruple} $\mathbf v$, which corresponds to the four largest circles in the packing, and is uniquely determined for any given ACP.  From this point on, we denote by $\mathbf v$ the root quadruple of the packing in question, and assume that the packing $P$ is primitive -- the gcd of the coordinates of any $w\in A\mathbf v$ in the orbit is $1$.  The group $A$, called the Apollonian group, was introduced by Hirst in 1967 (see \cite{Hirst}).  It is a group on the four generators
 \begin{equation}\label{gens}
 \small{
S_1=\left(
\begin{array}{llll}
-1&2&2&2\\
0&1&0&0\\
0&0&1&0\\
0&0&0&1\\
\end{array}
\right)\quad
S_2=\left(
\begin{array}{llll}
1&0&0&0\\
2&-1&2&2\\
0&0&1&0\\
0&0&0&1\\
\end{array}
\right)}
\end{equation}
\begin{equation*}
 \small{
S_3=\left(
\begin{array}{cccc}
1&0&0&0\\
0&1&0&0\\
2&2&-1&2\\
0&0&0&1\\
\end{array}
\right)\quad
S_4=\left(
\begin{array}{cccc}
1&0&0&0\\
0&1&0&0\\
0&0&1&0\\
2&2&2&-1\\
\end{array}
\right),}
\end{equation*}
derived by fixing all but one of the $x_i$ in (\ref{descartes}) and solving $Q(x_i)=0$ for the fourth.  Note that each $S_i$ is of order $2$ and determinant $-1$, and that left multiplication of $S_i$ by a vector $\mathbf w=(w_1,w_2,w_3,w_4)^{\textrm T}\in \mathbb Z^4$ fixes three of the coordinates of $\mathbf w$.  We denote by $\mathcal O = A \mathbf v$ the orbit of $A$ acting on $\mathbf v$, and note that our results hold for an arbitrary $\mathbf v$, as they are based on properties of the group $A$, and do not depend on the specific orbit in question.  Since the quadratic form in (\ref{descartes}) is of signature $(3,1)$ over $\mathbb R$, we have that $A$ is a subgroup of $\textrm{O}(3,1)$ and can be thought of as a subgroup of the group of motions of hyperbolic $3$-space $\mathbb H^3$.  In this way $A$ is a discrete algebraic group acting on $\mathbb H^3$ where the complement of three mutually tangent hemispheres inside an infinite cylinder is the fundamental domain of the action.  This fundamental domain has infinite volume, which makes counting integers in the group's orbit quite difficult.  We note, however, that $A$ contains Fuchsian triangle subgroups generated by any three of the $S_i$ above, which are lattices in the corresponding $\textrm{O}(2,1)$'s.  We use this fact extensively throughout this paper.

To this end, denote by $A_i$ the subgroup of $A$ generated by three of the four generators as follows:
$$A_i:= \left(\{S_1,S_2,S_3,S_4\}-\{S_i\}\right).$$
This group is the Schottky group generated by reflections in the three circles intersecting the $i$th circle in the root quadruple and perpendicular to the initial circles in the packing; in particular, the $i$th circle is fixed under this action.  The fundamental domain of $A_i$ is then a triangle bounded by the three circles, and has hyperbolic area $\pi$.

\section{A preliminary lower bound}\label{prlow}

In this section, we follow \cite{sl} in order to count integer points in an orbit of a subgroup $A_i$ of the Apollonian group as described in Section~\ref{agroup}.  This produces a preliminary lower bound on the number $\kappa(P,X)$ of integers less than $X$ occurring as curvatures in an Apollonian packing $P$.
\begin{prop}\label{lowbound}
For an integer Apollonian circle packing $P$, let $\kappa(P, X)$ denote the number of distinct integers less than $X$ occurring as curvatures in the packing.  Then we have
$$\kappa(P, X) \gg {{X}\over{\sqrt{\log X}}}.$$
\end{prop}
\begin{proof}
We fix a circle $C_0$ of non-zero curvature $a_0$ in the packing $P$, and count the integers which occur as curvatures of circles tangent to $C_0$ in $P$.  This is identical to considering the orbit of $A_1$ acting on a quadruple $\mathbf v$ of mutually tangent circles $(a_0,b,c,d)$, since $A_1$ fixes the first coordinate of $\mathbf v$ and its orbit represents all of the circles tangent to $C_a$.  Note that $A$ generates all possible Descartes configurations in the packing $P$, and there can only be finitely many circles of curvature $a_0$ in the packing since the total area of all the inscribed circles is bounded by the area of the outside circle.  Therefore it is reasonable to count the circles represented in the orbit of $A_1$, since they make up a positive fraction of all of the circles in $P$ tangent to a circle of curvature $a_0$.

In this orbit, we have that the first coordinate $a_0$ is fixed, and the other coordinates of points in the orbit of $A_1$ vary to satisfy
$$Q(a_0,x_2, x_3, x_4)=2(a_0^2+x_2^2+x_3^2+x_4^2)-(a_0+x_2+x_3+x_4)^2=0,$$
where $Q$ is the Descartes form in (\ref{descartes}).  A change of variables $\mathbf y=(y_2,y_3,y_4)=(x_2,x_3,x_4)+(a_0,a_0,a_0)$ allows us to rewrite the equation above as
\begin{equation}\label{firstchange}
g(\mathbf y)+4a_0^2=0,
\end{equation}
where $g(\mathbf y)= y_2^2+y_3^2+y_4^2-2y_2y_3-2y_2y_4-2y_3y_4$ is the resulting ternary quadratic form.  We can thus conjugate the action of $A_1$ on $(a_0,x_2, x_3,x_4)$ to an action independent of $a_0$ which preserves the form $g$.  This is the action of a group $\Gamma$ on $\mathbf y$, generated by
\[\small{
\left(
\begin{array}{lll}
-1&2&2\\
0&1&0\\
0&0&1\\
\end{array}
\right),
\left(
\begin{array}{lll}
1&0&0\\
2&-1&2\\
0&0&1\\
\end{array}
\right),
\left(
\begin{array}{ccc}
1&0&0\\
0&1&0\\
2&2&-1\\
\end{array}
\right).}
\]
Moreover, the action of $\Gamma$ on
$$\mathbf v'=(b+a_0,c+a_0,d+a_0)$$
is related to the action of $A_1$ on $\mathbf v$ by
$$A_1\mathbf v=(a_0,\Gamma \,[\mathbf v'-(a_0,a_0,a_0)]),$$
so we count the same number of curvatures occurring in the packing before and after this change of variables.  We change variables once again by letting
$$y_2=A,\;y_3=A+C-2B,\; y_4=C.$$
We note that $(y_2,y_3,y_4)\in\mathbb Z^3$ implies that $A,B,$ and $C$ are integers, and the primitivity of the packing is preserved as well -- the gcd of $A, B$, and $C$ is $1$.  With this change of variables, $\Gamma$ is conjugated to an action of a group $\Gamma'$ on $(A,B,C)$ which is generated by
\[\small{
\left(
\begin{array}{lll}
1&-4&4\\
0&-1&2\\
0&0&1\\
\end{array}
\right),
\left(
\begin{array}{lll}
1&0&0\\
0&-1&0\\
0&0&1\\
\end{array}
\right),
\left(
\begin{array}{ccc}
1&0&0\\
2&-1&0\\
4&-4&1\\
\end{array}
\right).}
\]
Under this change of variables, the expression in (\ref{firstchange}) becomes
\begin{equation}\label{secondchange}
4(B^2-AC)=-4a_0^2.
\end{equation}
Letting $\Delta(A,B,C)$ denote the discriminant of the binary quadratic form $Ax^2+2Bxy+Cy^2$, (\ref{secondchange}) is simply  
$$\Delta(A,B,C)=a_0^2,$$
and thus $\Gamma'$ is a subgroup of $\textrm{O}_{\Delta}(\mathbb Z)$, the orthogonal group preserving $\Delta$.  Let $\tilde{\Gamma}$ denote the intersection $\Gamma'\cap \textrm{SO}_{\Delta}(\mathbb Z)$.  The spin double cover of $\textrm{SO}_{\Delta}$ is well known (see \cite{Elstrodt}) to be $\textrm{PGL}_2$, and is obtained via the homomorphism
\begin{equation}\label{homo}
\begin{array}{lcl}\rho:\textrm{PGL}_2(\mathbb Z) &\longrightarrow& \textrm{SO}_{\Delta}(\mathbb Z)\\
{\left(\begin{array}{ll}\alpha&\beta\\ \gamma&\delta\\ \end{array}\right)}&{\stackrel{\rho}{\longmapsto}}&{\frac{1}{\alpha\delta-\beta\gamma}\cdot\left(\begin{array}{ccc}\alpha^2&2\alpha\gamma&\gamma^2\\ \alpha\beta&\alpha\delta+\beta\gamma&\gamma\delta\\ \beta^2&2\beta\delta&\gamma^2\\ \end{array}\right)}\\
\end{array}
\end{equation}
written here over $\mathbb Z$ as this is the situation we work with.  It is natural to ask for the preimage of $\tilde{\Gamma}$ under $\rho$ which we determine in the following lemma.
\begin{lemma}\label{intersection}
Let $\tilde\Gamma$ and $\rho$ be as before.  Let $\Lambda(2)$ be the congruence $2$-subgroup of $\textrm{PSL}_2(\mathbb Z)$.  Then the preimage of $\tilde{\Gamma}$ in $\textrm{PGL}_2(\mathbb Z)$ under $\rho$ is $\Lambda(2)$.
\end{lemma}
\begin{proof}
We can extract from the generators of ${\Gamma}'$ as well as the formula in (\ref{homo}) that the preimage of $\tilde{\Gamma}$ under $\rho$ contains
\[
\left(
\begin{array}{cc}
1&0\\
-2&1\\
\end{array}
\right)
\mbox{ and }
\left(
\begin{array}{cc}
1&-2\\
0&1\\
\end{array}
\right),
\]
and so $\tilde{\Gamma}$ contains the congruence subgroup $\Lambda(2)$ of $\textrm{SL}_2(\mathbb Z)$.  Recall that the area of $A_1\backslash \mathbb H^2$ is $\pi$, and note that $\textrm{SO}_{\Delta}(\mathbb Z)\bigcap \Gamma'$ contains exactly those elements of $\Gamma'$ which have even word length when written via the generators of $\Gamma'$, making up half of the whole group.  Therefore the area of $\tilde{\Gamma}\backslash \mathbb H^2$ is $2\pi$, which is equal to the area of $\Lambda(2)\backslash \mathbb H^2$, and hence the preimage of $\tilde{\Gamma}$ in $\textrm{PGL}_2(\mathbb Z)$ is precisely $\Lambda(2)$ as desired.
\end{proof}
Recall that we would like to count the integer values of $y_2,y_3$, and $y_4$ -- in terms of the action of $\Gamma$, we are interested in the set of values $A, C,$ and $A+C-2B$ above.  Lemma~\ref{intersection} implies that these values contain integers represented by the binary quadratic form
\begin{equation}\label{binquad}
f_{a_0}(\zeta,\nu)= A_0\zeta^2+2B_0\zeta\nu+C_0\nu^2,
\end{equation}
where $(\zeta,\nu)=1$, and the coefficients are derived from the change of variables above:
\begin{equation}\label{coefficients}
A_0=b+a_0,\; C_0=d+a_0,\;B_0={{b+d-c}\over 2}.
\end{equation}
We note that the discriminant of this form is not a square, since
\begin{equation}\label{disc}
D(f_{a_0})= (2B_0)^2-4A_0C_0 = -4a_0^2
\end{equation}
and $a_0\not=0$.  Since the vectors in the orbit of $A_1$ are of the form $(a_0,A-a_0, A+C-2B-a_0,C-a_0)$, they correspond to the integer values of
\begin{equation}\label{shifta}
f_{a_0}(x,y)-a_0,
\end{equation}
where $f_{a_0}$ is as before.  Therefore
\begin{equation}\label{relatekappa}
\kappa(P,X)\gg \#\{m\in \mathbb Z\,|\,m>0,\, f_{a_0}(x,y)-a_0 = m {\mbox{ for some }} x,y\in\mathbb Z, (x,y)=1\}
\end{equation}
and we need only to count the integers represented by $f_{a_0}$ in order to get a bound on the number of curvatures in $P$.  This is done in \cite{james}, from which we have the following:
\begin{lemma}(James): \label{density}
Let $f$ be a positive definite binary quadratic form over $\mathbb Z$ of discriminant $-D$, where $D$ is a positive integer.  Denote by $B_D(X)$ the number of integers less than $X$ represented by $f$.  Then
$$B_D(X)= {{c\cdot X}\over{\sqrt{\log X}}}+O\left({X\over{\log X}}\right),$$
where
$$\pi c^2 = \prod_{\stackrel{q\equiv 3 \; (4)}{q\not| D}}\left(1-{1\over {q^2}}\right)^{-1}\prod_{p|D}\left(1-{1\over p}\right)\sum_{n=1}^{\infty}\left({{-D}\over n}\right)n^{-1}.$$
\end{lemma}
Lemma~\ref{density} paired with (\ref{relatekappa}) implies that
$$\kappa(P,X)\gg \frac{X}{\sqrt{\log X}}$$
as desired.
\end{proof}

\section{Proof of Theorem~\ref{iteratedens}}\label{iterate}
In this section we sharpen the bound in Proposition~\ref{lowbound} in order to answer the question posed by Graham et.al. in \cite{Apollo} and prove that the integers appearing as curvatures in any integer ACP make up a positive fraction of all positive integers.  Our computation in Section~\ref{prlow} reflects only those circles which are tangent to a fixed circle in $P$.  It is thus natural to count some of the omitted curvatures here.  Specifically, we repeat the method from Section~\ref{prlow} several times, fixing a different circle $C$ each time and counting the integers occuring as curvatures of circles tangent to $C$.

Recall that to prove Proposition~\ref{lowbound} we fixed a circle of curvature $a_0$, and associated curvatures of circles tangent to it with the set of integers (without multiplicity) represented by $f_{a_0}(x,y)-{a_0}$.  We denote the set of these integers that are less than $X$ by $\mathcal A_0$:
$$\mathcal A_0 = \{a\in\mathbb N\, | \, a\leq X, f_{a_0}(x,y)-a_0 = a \mbox{ for some integers } x,y\geq 0\}$$
For every $a\in\mathcal A_0$ we use the method in Section~\ref{prlow} to produce another shifted binary quadratic form
$$f_a(x,y)-a$$
of discriminant $-D=-4a^2$.  As in Section~\ref{prlow}, we wish to count the integers represented by these new forms.  For each $a\in\mathcal A_0$, let $S_a$ denote the set of integers less than $X$ represented by $f_a-a$:
$$S_a=\{n\in \mathbb N\, |\, n\leq X, n=f_a(x,y)-a\; \mbox{ for some relatively prime integers }x,y\geq 0\}$$
Note that the sets $S_a$ depend only on $a_0$, the curvature of $C_0$.  One important consideration in counting the integers represented by the forms $f_a$ is that their discriminants can be very large with respect to $X$, and thus many of the represented integers may be $>X$.  In particular, the count in Lemma~\ref{density} is not uniform in $D$ so we use more recent results of Blomer and Granville in \cite{bg} which specify how the number of integers less than $X$ represented by a binary quadratic form depends on the size of the discriminant of the form\footnote[2]{The results of Blomer and Granville concern quadratic forms of square-free determinant, but the authors note in section 9.3 of \cite{bg} that the same can be done for binary quadratic forms of non-fundamental determinant as well.}.

With this notation, the bounds in \cite{bg} yield a {\it lower} bound on $\sum_{a}|S_a|$ for the $a$'s we consider.  We also compute an {\it upper} bound on $\sum_{a,a'}|S_a\cap S_{a'}|$ for $a\not=a'$ so that
$$\sum_a |S_a|-\sum_{a,a'}|S_a\cap S_{a'}|$$
gives a lower bound for $\kappa(P,X)$.  A crucial ingredient to computing this and proving Theorem~\ref{iteratedens} (that the integers appearing as curvatures in a given ACP make up a positive fraction of $\mathbb N$) is the balance between these lower and upper bounds -- for example, the more sets $S_a$ we choose to include in our count, the bigger the lower bound on $\sum_a |S_a|$.  However, choosing too many such sets will also increase the upper bound on the second sum $\sum_{a,a'}|S_a\cap S_{a'}|$.  In fact, it is possible to choose so many sets $S_a$ that the upper bound on the intersections outweighs the lower bound on the sizes of $S_a$.  In Section~\ref{union} we specify how we choose the $a$'s used in our computation, and compute the first sum, $\sum_a |S_a|$.  In Section~\ref{intersect}, we compute an upper bound on $\sum_{a,a'}|S_a\cap S_{a'}|$ for $a\not=a'$ to prove Theorem~\ref{iteratedens}.

\subsection{Integers represented by multiple binary quadratic forms}\label{union}

In this section, we evaluate the sum $\sum_a |S_a|$, choosing $a$'s in a subset of $\mathcal A_0$ in order to ensure that we obtain a positive fraction of $X$ in our final count.  Specifically, we consider $a\in \mathcal A_0$ such that
$$(\log X)^2\leq a \leq (\log X)^3$$
This interval is chosen to give us the desired lower bounds in conjunction with results in \cite{bg} -- this will become clear in the computations preceding (\ref{blomgranderive}).  We would like to further reduce the set of $a$'s we consider so that the bounds on the size of the intersections of sets $S_a$ are not too large.  To do this, we first partition the interval $[(\log X)^2,(\log X)^3]$ into dyadic ranges $[2^k, 2^{k+1}]$ and select $a$'s within these ranges.

Namely, we consider $\mathcal A_0\cap [2^k, 2^{k+1}]$ where $(\log X)^2\leq 2^k, 2^{k+1} \leq (\log X)^3$.  The size of this set depends only on $a_0$, the curvature of the original circle we fixed.  By Lemma~\ref{density}, we have
\begin{equation}\label{sizeA0}
\left|\mathcal A_0\cap [2^k, 2^{k+1}]\right| \gg \frac{2^k}{\sqrt{k}}
\end{equation}
where the implied constant depends on $a_0$.  We partition each dyadic interval $[2^k,2^{k+1}]$ into intervals $[2^k+n\cdot\eta\frac{2^k}{\sqrt k}, 2^k+(n+1)\cdot\eta\frac{2^k}{\sqrt k}]$ of length $\eta\frac{2^k}{\sqrt k}$, where $0<\eta<1$ is a fixed parameter whose importance will become apparent in Proposition~\ref{intersectionlemma}.  We note that the average over $0\leq n \leq \sqrt{k}\eta^{-1}$ of cardinalities of the corresponding subsets of $\mathcal A_0$ is
\begin{equation}\label{intersectabove}
E_{n}\left(\Bigl|\mathcal A_0\cap [2^k+n\cdot\eta\frac{2^k}{\sqrt k}, 2^k+(n+1)\cdot\eta\frac{2^k}{\sqrt k}]\Bigr|\right)\gg \eta \frac{2^k}{k}
\end{equation}
by (\ref{sizeA0}).  Thus for every value of $k$ there exists an $0\leq n \leq \sqrt{k}\eta^{-1}$ for which the intersection in (\ref{intersectabove}) contains $\gg \eta \frac{2^k}{k}$ integers.  For simplicity of notation, we assume without loss of generality\footnote[1]{One can in fact extend Lemma~\ref{density} to show that this holds for \emph{every} $n$.  Friedlander and Iwaniec do this for $a_0=1$ in Theorem 14.4 of \cite{FI}.  However, it is not necessary here.} that $n=0$, and define $\mathcal A^{(k)}$ to be
\begin{equation}\label{defineAk}
\mathcal A^{(k)} = \mathcal A_0 \cap [2^k, 2^k+\eta\frac{2^k}{\sqrt k}]
\end{equation}
where we have
\begin{equation}\label{sizeAk}
|\mathcal A^{(k)}|= \eta \frac{2^k}{k}
\end{equation}
up to a constant which depends only on $a_0$.  Denote the union of these subsets by $\mathcal A$:
\begin{equation}
\mathcal A=\bigcup\mathcal A^{(k)}
\end{equation}
The results in \cite{bg} imply the following lemma regarding the integers represented by quadratic forms associated with $a\in \mathcal A$.
\begin{lemma}\label{sumsa1}
Let $\mathcal A$ and $f_a$ be as before.  Then we have
\begin{equation*}
\sum_{a\in \mathcal A}|S_a|\quad\gg\quad\eta X
\end{equation*}
\end{lemma}
To prove Lemma~\ref{sumsa1}, we recall the notation and relevant theorem from \cite{bg}.  Let $f$ be a binary quadratic form of discriminant $-D$, and let $r_f(n)$ be the number of representations of $n$ by $f$:
\begin{equation}\label{repnumber}
r_f(n) = \#\{(m_1,m_2)\in \mathbb Z^2-\{\mathbf 0\}\, |\; gcd(m_1,m_2)=1, \,f(m_1,m_2) = n\}
\end{equation}
Let $U_f^{0}(X) = \sum_{n\leq X} r_f(n)^{0}$, the number of integers less than $X$ represented by $f$, counting without multiplicity.  In \cite{bg}, Blomer and Granville compute bounds for $U_f^{0}(X)$ for $D$ in three ranges between $0$ and $X$.  These ranges are defined in terms of the class number $h$ of the binary quadratic form $f$ and by $g$, the number of genera.  Letting $\ell = \ell_{-D}=L(1,\chi_{-D})(\phi(D)/D)$, they create a parameter
$$\kappa = \frac{\log(h/g)}{(\log 2)(\log(\ell_{-D}\log X))}$$
where $h/g = D^{\nicefrac{1}{2}+\textrm{o}(1)}$.   Their bounds for $U_f^{0}(X)$ are then uniform in $D$ for each range below (see Lemma~\ref{blomgran}):
\begin{itemize}
\item $0\leq \kappa \leq \nicefrac{1}{2}$
\item $\nicefrac{1}{2} <\kappa < 1$
\item $1\leq \kappa \ll \frac{\log D}{\log\log D}$
\end{itemize}
In the first and last range, they are able to compute both an upper and lower bound on $U$.  However they prove only an upper bound for $U_f^{0}(X)$ in the case that $D$ is in the middle range, which is not suitable for our purposes.  The lower bound for $U_f^0(X)$ for a form $f$ of discriminant $-D$ where $D$ is in the smallest range is essentially James' result in Lemma~\ref{density}, and is used to show that
$$\kappa(P,X)\gg \frac{X}{(\log X)^{\epsilon}}$$
in \cite{Fuchspreprint}.  Our results and the statement in Lemma~\ref{sumsa1} depend on Blomer and Granville's lower bound for $U_f^0(X)$ where $f$ is of dicriminant $-D$ and $D$ is in the third range above.  Specifically, we use Theorem 2 from \cite{bg}, which is summarized in the lemma below.
\begin{lemma}(Blomer, Granville):\label{blomgran}
Let $f$ be a binary quadratic form of discriminant $-D$, and let $U_f^0(X)$ be as before.  Let $\mathcal G$ be the group of genera of binary quadratic forms of discriminant $-D$.  Denote by $s$ the smallest positive integer that is represented by $f$, and by $u$ the smallest positive integer represented by some form in the coset $f\mathcal G$.  Then
\begin{equation}\label{boundit}
U_f^0(X) = \pi\cdot \left(1-\frac{1}{2u}\right)\cdot \frac{X}{\sqrt D} + E_0(X, D)
\end{equation}
where
\begin{equation}\label{errorit}
E_0(X,D) \ll \sqrt{\frac{X}{s}} + \tau(D)\cdot \left(\frac{X\log X}{D}+\frac{X}{D^{\frac{3}{4}}}\right)
\end{equation}
where $\tau(D)$ is the number of prime divisors of $D$, and the implied constant does not depend on $D$.
\end{lemma}
With this in mind we are ready to prove Lemma~\ref{sumsa1}.

\noindent{\it Proof of Lemma~\ref{sumsa1}:}

We use Lemma~\ref{blomgran} to count the integers less than $X$ represented by forms of discriminant $-D$ where $D$ is a power of $\log X$ in our case.   Recall that
$$f_a(x,y)=\alpha x^2+2\beta xy +\gamma y^2$$
is of discriminant $-D=-4a^2$.  In particular, since $(\log X)^2\leq a\leq (\log X)^3$, we have that
$$(\log X)^4\leq D\leq (\log X)^6$$
and the number of prime divisors of $D$ is
$$\tau(D) \ll \log \log X,$$
and so
$$E_0(X,D) \ll \frac{X}{(\log X)^3}+(\log D)\cdot\frac{X}{D^{\frac{3}{4}}}$$
Thus we have that the error $E_0(X,D)\ll \frac{X}{D^{\nicefrac{3}{4}-\epsilon}}$ for any $\epsilon>0$, and thus Lemma~\ref{blomgran} implies
\begin{equation}\label{blomgranderive}
U_f^0(X) \gg \frac{X}{\sqrt D}
\end{equation}
where the implied constant does not depend on $D$.  Since $D=-4a^2$, it follows from (\ref{blomgranderive}) that the number of distinct values less than $X$ represented by $f_a$ is $\gg \frac{X}{a}$ and we have
\begin{eqnarray}\label{sumsa}
\sum_{a\in \mathcal A}|S_a|&\gg& \sum_{a\in \mathcal A}\frac{X}{a}\nonumber \\
&\gg&\eta\cdot X\cdot\hspace{-0.2in}\sum_{\stackrel{2^k<(\log X)^3}{2^k>(\log X)^2}}\frac{1}{k}\nonumber\\
&\gg&\eta\cdot X
\end{eqnarray}
as desired.
\qed

The sum in Lemma~\ref{sumsa1} is the lower bound on the number of integers we count by considering the quadratic forms associated with $a\in\mathcal A$.  In order to prove Theorem~\ref{iteratedens}, we obtain an upper bound on the number of integers we have counted twice in this way in the next section.
\subsection{Integers in the intersections}\label{intersect}
To prove Theorem~\ref{iteratedens} we would like to show
\begin{equation}\label{uniondef}
\left|\bigcup_{a\in\mathcal A}S_a\right|\gg X
\end{equation}
since this union is a subset of all curvatures less than $X$ in the packing $P$.  By Lemma~\ref{sumsa1}, we may estimate the size of this union as follows:
\begin{eqnarray}\label{unionbig}
\left|\bigcup_{a\in\mathcal A}S_a\right|&\geq& \sum_{a\in \mathcal A}|S_a|-\sum_{a\not=a'\in\mathcal A}\left|S_a\cap  S_{a'} \right|\nonumber\\
&\gg& \eta X -\sum_{a\not=a'\in\mathcal A}\left|S_a\cap  S_{a'} \right|
\end{eqnarray}
We need only to determine an upper bound for the last sum above.  We do this by counting points $(x,y,x'y')$ in a box on the quadric
$$f_a(x,y)-f_{a'}(x',y')=a'-a$$
for each $a\not=a'\in\mathcal A$.  The region in which we count these points is induced by the condition that $f_a(x,y)<X$.  Namely, rewriting the binary form $f_a$ as
\begin{equation}\label{rewriteform}
f_a(x,y)=\frac{(\alpha x+\beta y)^2 +4 a^2 y^2}{\alpha}
\end{equation}
we can define a region
\begin{equation}\label{ball}
B_a=\{(x,y)\in\mathbb R^2\mbox{ s.t. }\,|\alpha x+\beta y|\ll\sqrt{|\alpha|}\mbox{ and } \, |y|\ll\frac{\sqrt{|\alpha|}}{a}\}
\end{equation}
so that $f_a(x,y)\ll 1$ for $(x,y)\in B_a$, and $f_a(x,y)\ll X$ for every $(x,y)\in \sqrt{X}\,B_a$ as desired.  Therefore, the region in $\mathbb R^4$ over which we consider the forms $f_a-f_{a'}$ will be
$$\mathcal B_{a,a'} = (\sqrt{X}B_a\times \sqrt{X}B_{a'})\cap \mathbb Z^4$$
With this notation, we are ready to prove the following proposition.
\begin{prop}\label{intersectionlemma}
Let $\mathcal A$, $S_a$, $S_{a'}$, $\eta$, and $X$ be as before.  Then there exists $c>0$ depending only on $a_0$ such that
\begin{equation}\label{intersum}
\sum_{a\not=a'\in\mathcal A}\left|S_a\cap  S_{a'} \right|\leq c\eta^2X
\end{equation}
\end{prop}
Note that, since we chose $0<\eta<1$, we have $\eta^2<\eta$, and so this upper bound on the size of the intersection of the sets $S_a-a$ is small compared to the count in Lemma~\ref{sumsa1}.
\begin{proof}
We note that the expression inside the sum has an upper bound
\begin{eqnarray}\label{upint}
&&\left|S_a\cap S_{a'} \right|\\ &\leq& |\{(x,y,x',y')\in\mathcal{B}_{a,a'}\,|\, f_a(x,y)-f_{a'}(x',y')=a-a'\}|\nonumber
\end{eqnarray}
Although bounding (\ref{intersum}) in this way involves counting the integers in $S_a\cap S_{a'} $ with multiplicity, our analysis shows that this sacrifice is in fact not too expensive to our final count.  We thus consider the quaternary quadratic form
$$F(x,y,x',y')=f_a(x,y)-f_{a'}(x',y')$$
with discriminant $\Delta=(\beta^2-\alpha\gamma)(\beta'^2-\alpha'\gamma')=16a^2(a')^2$.  To obtain an upper bound on the number of points in $\mathbf x\in\mathcal B_{a,a'}$ for which $F(\mathbf x)=a'-a$, one can use the well developed circle method following Kloosterman in \cite{Kloo} and Esterman in \cite{Est} or modular forms (see \cite{drs}).  Both methods would yield what we want -- the latter would give the best results but is not as flexible as the former for our purposes since we wish to vary the parameters $a$ and $a'$ which is more straightforward in the circle method.  Heath-Brown's Theorem 4 in \cite{HB} and Niedermowwe's  Theorem 5.6 in \cite{Niedermowwe} determine representation numbers of a fixed indefinite quadratic form\footnote[3]{Note that in \cite{HB} one considers representations of an integer $m$ by $F$ where $m$ is asymptotic to the scaling factor $P$ of the unscaled domain $B$ (in our case $B=B_a\times B_{a'}$ and $P=\sqrt{X}$), while in \cite{Niedermowwe} $m$ is any nonzero integer.}.  Since our $a,a'$ are all a small power of $\log X$, the proofs of these theorems can be manipulated slightly to yield the following lemma regarding representation numbers of all the indefinite quaternary quadratic forms we consider:
\begin{lemma}\label{niederheath}
Let $F$ be as before, and let $(\log X)^2\leq a,a'\leq (\log X)^3$.  Let $\chi_{a,a'}$ denote the characteristic function on the region $\mathcal{B}_{a,a'}$, and let
$$R_{\chi_{a,a'}}(a-a') = \sum_{\stackrel{\mathbf{x}\in \mathbb Z^4}{F(\mathbf x)=a-a'}}\chi_{a,a'}(\mathbf x).$$
Let $\Delta$ be as above.  Then we have
\begin{eqnarray}\label{errortermlabel}
&&R_{\chi_{a,a'}}(a-a')\nonumber \\
&=&|I_{\chi_{a,a'}}(a-a')|\cdot| \mathfrak S(a-a')| + O\left(\frac{X\cdot \Delta^{100}}{(\log X)^{\lambda}}\right)
\end{eqnarray}
where the first factor is the singular integral
\begin{equation}\label{singint}
I_{\chi_{a,a'}}=\int_{-\infty}^{\infty}\Biggl[\int_{\mathbb R^4}\chi_{a,a'}(x)e(z(F(x)-a+a'))dx\Biggr]dz
\end{equation}
and the second factor is the singular series
\begin{equation}\label{singser}
\mathfrak S(a-a')=\prod_p\sigma_p
\end{equation}
where
\begin{equation}\label{sigma}
\sigma_p=\lim_{k\rightarrow\infty}p^{-3k}\cdot\#\{\mathbf x\in(\nicefrac{\mathbb Z}{p\mathbb Z})^4\mbox{ s.t. }\, F(\mathbf x)\equiv a-a'\,(p^k)\}
\end{equation}
and $e(z)=e^{2\pi i z}$.
\end{lemma}
In the error term in (\ref{errortermlabel}), $\lambda$ is an arbitrary large fixed constant.  With more effort we can in fact get a power saving here by using modular forms -- while this would yield the best result, the methods in \cite{HB} and \cite{Niedermowwe} suffice.  In particular, the argument in \cite{Niedermowwe} lends itself well to our consideration of the quadratic form $F$, which has a discriminant of size $(\log X)^k$.  The error term in Niedermowwe's Theorem 5.6 consists of a power saving in $X$, and a careful examination of the proof shows that the dependency on the discriminant of the form is absorbed into the error term since it is only logarithmically large -- this is reflected in (\ref{errortermlabel}) via a power of the discriminant $\Delta$ of $F$.  It is similarly important here that the distorsion of $B_a$ and $B_{a'}$ with respect to the standard cube discussed in \cite{HB} and \cite{Niedermowwe} is logarithmic in $X$.

To prove Proposition~\ref{intersectionlemma} it remains to evaluate the singular integral and singular series in (\ref{singint}) and (\ref{singser}).  For a set $P\subset \mathbb R^4$, let $\mathbf V(P)$ denote the measure of $P$.  From the definition of $f_a$ and $B_a$ in (\ref{rewriteform}) and (\ref{ball}), we have
\begin{eqnarray}\label{integral}
I_{\chi_{a,a'}}
&\ll&\lim_{\epsilon\rightarrow 0}\,\frac{1}{\epsilon} \cdot \mathbf V\biggl(\{(x,y,x',y')\in \sqrt{X}B_a\times\sqrt{X}B_{a'}\,|\, |f_a(x,y)-f_{a'}(x',y')-a+a'|<\epsilon\}\biggr)\nonumber\\
&\ll&\lim_{\epsilon\rightarrow 0}\,\frac{1}{\epsilon}\cdot\frac{\epsilon}{\sqrt{|\alpha|X}}\cdot\frac{\sqrt{|\alpha|X}}{a}\cdot\sqrt{\frac{X}{|\alpha'|}}\cdot\frac{\sqrt{|\alpha'|X}}{a'}\nonumber\\
&\ll&\frac{X}{aa'}
\end{eqnarray}
To evaluate the singular series $\mathfrak S(a-a')$ we prove the following lemma.
\begin{lemma}\label{singseries}
Let $\mathfrak S(a-a')$ be the singular series defined in (\ref{singser}).  We have
$$\mathfrak S(a-a')\ll \prod_{\stackrel{p|aa'(a-a')}{p\not| (a,a')}}\Bigl(1+\frac{1}{p}\Bigr)\cdot 2^{\omega((a,a'))}$$
\end{lemma}
\begin{proof}
We compute an upper bound for the expression in the limit in (\ref{sigma}) by letting $k=1$ since the expression in the limit decreases with $k$.  Note that if $p|(a,a')$, we have $F$ is not degenerate modulo $p$ by the primitivity of the packing and the definition of the coefficients of $f_a$ in (\ref{coefficients}).  Therefore
$$\sigma_p<p^{-3}\cdot\#\{\mathbf x\in(\nicefrac{\mathbb Z}{p\mathbb Z})^4\mbox{ s.t. }\, F(\mathbf x)\equiv 0\; (p)\},$$
and over $\mathbb F_p$, the number of nontrivial representations of $0$ by $F$ is bounded above by $2p^3$ (see \cite{Cassels}, for example), so $\sigma_p$ is bounded above by $2$ in this case.  In the other cases, we use exponential sum estimates.  Taking $k=1$ as before, we have
$$\sigma_p=\frac{1}{p}\sum_{r=0}^{p-1}\Bigl[\sum_{x,y}e_p(rf_a(x,y))\Bigr]\Bigl[\sum_{x',y'}e_p(-rf_{a'}(x',y'))\Bigr]e_p(r(a-a'))$$
where $e_p(z)=\exp(\frac{2\pi i z}{p})$.  There are several cases to consider:
\vspace{0.2in}

\noindent{\it Case 1: $p$ does not divide $aa'(a-a')$}:

If we diagonalize $f_a$ and $f_{a'}$, we obtain
$$\sigma_p= p^3+\frac{1}{p}\sum_{r=1}^{p-1}\Biggl(\frac{\tilde{\alpha}r}{p}\Biggr)^2\Biggl(\frac{\tilde{\alpha}'r}{p}\Biggr)^2p^2e_p(r(a-a'))=p^3+\textrm{o}(p)$$
since $(a-a',p)=1$.
\vspace{0.2in}

\noindent{\it Case 2: $p|a-a'$ and does not divide $aa'$}:

In this case we have $\sigma_p<p^3+\textrm{o}(p^2)$.
\vspace{0.2in}

\noindent{\it Case 3: $p|a$ and $p\not| a'$}:

Diagonalizing $f_{a'}$, we obtain
$$\sigma_p=p^3+\frac{1}{p}\sum_{r=1}^{p-1}\Biggl(\frac{\alpha r}{p}\Biggr)p\sqrt{p}\cdot p\cdot e_p(r(a-a'))<p^3+\textrm{o}(p^2)$$
From these bounds and Lemma~\ref{sievebound}, we obtain the desired result in Lemma~\ref{singseries}.
\end{proof}
Combining our computation of the singular integral in (\ref{integral}) and the bound on the singular series in Lemma~\ref{singseries}, the result of Niedermowwe in Lemma~\ref{niederheath} yields
\begin{equation}\label{referback}
|S_a\cap S_{a'} |\ll \frac{X}{aa'}\cdot\prod_{\stackrel{p|aa'(a-a')}{p\not| (a,a')}}\Bigl(1+\frac{1}{p}\Bigr)\cdot 2^{\omega((a,a'))}
\end{equation}
where $\omega(n)$ is the number of distinct prime factors of $n$.  Thus to evaluate the last sum in (\ref{unionbig}), we count the number of $a\in \mathcal A$ in progressions $a\equiv r$ mod $q$.  To this end, we recall Theorem 14.5 from \cite{FI} of Friedlander and Iwaniec regarding sums of squares in progressions in the following lemma\footnote[4]{Note that the set of integers in the interval $[2^k, 2^{k}+\eta\frac{2^k}{\sqrt{k}}]$ which can be written as sums of two squares contains the $a\in \mathcal A^{(k)}$ in progressions $a\equiv r\, (q)$, since $\mathcal A^{(k)}$ is a set of integers represented by a binary quadratic form of discriminant $-\delta^2$.  This count is therefore an upper bound on what we want.}.
\begin{lemma}{(Friedlander, Iwaniec):}\label{friedithm}
Let $b(n)$ be a characteristic function defined as
\[
b(n)= \left\{ \begin{array}{ll}
                			1& \mbox{if $n = s^2+t^2$ for some $s,t\in\mathbb Z$} \\
                			0 & \mbox{otherwise} \\
                			  \end{array}
                		 \right. \\
\]
and let
$$B(x,q,a) = \sum_{\stackrel{n\leq x}{n\equiv a\, (q)}}b(n)$$
For $2\leq q\leq x$, $(a,q=1)$, and $a\equiv 1$ mod $(4,q)$ we have
$$B(x,q,a) = \frac{c_q}{q}\cdot \frac{x}{\sqrt{\log x}}\Biggl[1+\textrm{O}\Biggl[\left(\frac{\log q}{\log x}\right)^{\frac{1}{7}}\Biggr)\Biggr]$$
where the implied constant is absolute and $c_q\ll \log\log q$ is a positive constant.
\end{lemma}
We note that the statement in Lemma~\ref{friedithm} is much stronger than what we need -- we require only an upper bound on $B(x,q,a)$, which could be proven using an upper bound sieve.  Since our set $\mathcal A$ is obtained via the fixed quadratic form of discriminant $-4a_0^2$ from Section~\ref{prlow}, such an upper bound implies the following in our case.
\begin{lemma}\label{sievebound}
Let $\mathcal A$, $X$, and $\eta$ be as before.  Then we have
\begin{equation*}
\sum_{\stackrel {a\in\mathcal A}{a\equiv r\,(q)}} \frac{1}{a}\ll \frac{\log\log q}{q}\cdot \eta
\end{equation*}
where $1<q<\log X$ is a square-free integer.
\end{lemma}
\begin{proof}
With the definition of $\mathcal A^{(k)}$ in (\ref{defineAk}), we may bound above the sum in Lemma~\ref{sievebound} as a sum over $k$ for which $(\log X)^2\leq 2^k, 2^{k+1}\leq (\log X)^3$:
\begin{equation}\label{atok}
\sum_k\frac{1}{2^k}\sum_{\overset{a\in \mathcal A_0}{\underset{a\equiv r\, (q)}{a\in[2^k,2^k+\eta\frac{2^k}{\sqrt{k}}]}}} 1
\end{equation}
By Lemma~\ref{friedithm}, the inner sum is bounded above (up to a constant) by
$$\eta\frac{c_q}{q}\frac{2^k}{k}\Biggl[1+\textrm{O}\Biggl(\left(\frac{\log q}{\log\log q}\right)^{\frac{1}{7}}\Biggr)\Biggr]$$
Since $c_q\ll \log\log q$, substituting this into (\ref{atok}) we have
$$\sum_k\eta\frac{c_q}{q{k}}\Biggl[1+\textrm{O}\Biggl(\left(\frac{\log q}{\log\log q}\right)^{\frac{1}{7}}\Biggr)\Biggr]\ll \eta\frac{\log\log q}{q}$$
as desired.
\end{proof}
With this in mind, we may evaluate the sum in (\ref{referback}) as follows.
\begin{eqnarray}
\sum_{a\not=a'\in\mathcal A}|S_a\cap S_{a'} |&\ll&X\cdot\sum_{a\not=a'\in\mathcal A}\frac{1}{aa'}2^{\omega((a,a'))}\prod_{\stackrel{p|aa'(a-a')}{p\not| (a,a')}}\Bigl(1+\frac{1}{p}\Bigr)\label{intline1}\\
&\ll&X\cdot\sum_{q_0,q_1,q_1',q_2}\frac{2^{\omega(q_2)}}{q_1q_1'q_2}\sum_{\overset{q_0q_1|a}{\underset{q_2|a-a'}{q_0q_1'|a'}}}\frac{1}{aa'}\label{intline2}
\end{eqnarray}
where $q_0,q_1,q_1',q_2$ are square-free and relatively prime.  We may restrict to primes $p<(\log X)^{\frac{1}{100}}$ in the product in (\ref{intline1}), we may restrict in (\ref{intline2}) the summation to $q_0,q_1,q_1',q_2<(\log X)^{\frac{1}{10}}$.  We bound the sum
$$\sum_{\overset{q_0q_1|a}{\underset{q_2|a-a'}{q_0q_1'|a'}}}\frac{1}{aa'}$$
using Lemma~\ref{sievebound}.  First fix $a$ and sum over $a'$ subject to the restrictions $q_0q_1'|a'$ and $a\equiv a'$ mod $q_2$.  From Lemma~\ref{sievebound}, we have
$$\sum_{\overset{a'\in\mathcal A}{\underset{q_2|a-a'}{q_0q_1'|a'}}}\frac{1}{a'}\ll \frac{\log\log (q_0q_1'q_2)}{q_0q_1'q_2}\cdot\eta$$
and
$$\sum_{q_0q_1|a}\frac{1}{a}\ll\frac{\log (q_0q_1)}{q_0q_1}\cdot \eta$$
so
\begin{equation}\label{finalint}
\sum_{\overset{q_0q_1|a}{\underset{q_2|a-a'}{q_0q_1'|a'}}}\frac{1}{aa'}\ll \frac{(\log\log(q_0+q_1+q_1'+q_2))^2}{q_0^2q_1q_1'q_2}\cdot\eta^2
\end{equation}
Substituting (\ref{finalint}) into (\ref{intline2}) gives the desired bound
\begin{eqnarray}
\sum_{a\not=a'\in\mathcal A}|S_a\cap S_{a'} |&\ll& \eta^2X\sum_{q_0,q_1,q_1',q_2}2^{\omega(q_2)}\cdot\frac{(\log\log(q_0+q_1+q_1'+q_2))^2}{(q_0q_1q_1'q_2)^2}\nonumber\\
&<&c\eta^2X
\end{eqnarray}
\end{proof}
Note that $\eta-\eta^2>0$ since $0<\eta<1$.We may take $\eta$ small enough so that (\ref{unionbig}) and (\ref{sumsa}) imply
$$\left|\bigcup_{a\in\mathcal A}S_a\right|\gg (\eta-c\eta^2)X\gg X$$
as desired.
\qed

We note that the methods used here are easily generalizable to many discrete linear algebraic groups acting on $\mathbb H^3$ with an integral orbit.  If the group contains several Fuchsian subgroups as in the case of the Apollonian group, we may restrict to the orbits of these subgroups as in Section~\ref{prlow}.  We would again utilize the subgroup's preimage in the spin double cover of $\textrm{SO}$ to relate the problem to integers represented by a binary quadratic form.  This would yield a comparable lower bound on the number of integers less than $X$ in the orbit of the group (counted without multiplicity).

\end{document}